\documentclass{amsproc}

\newtheorem{theorem}{Theorem}[section]
\newtheorem{lemma}[theorem]{Lemma}
\newtheorem{prop}[theorem]{Proposition}
\newtheorem{cor}[theorem]{Corollary}

\theoremstyle{definition}
\newtheorem{definition}[theorem]{Definition}

\theoremstyle{remark}
\newtheorem{remark}[theorem]{Remark}

\numberwithin{equation}{section}

\DeclareMathOperator{\Min}{Min}

\DeclareMathOperator{\Harm}{Harm}

\DeclareMathOperator{\GL}{GL}
\DeclareMathOperator{\Aut}{Aut}
\DeclareMathOperator{\Sp}{Sp}

\DeclareMathOperator{\SL}{SL}

\newcommand{\disj}{\stackrel{.}{\cup}}
\newcommand{\Z}{{\mathbb{Z}}}
\newcommand{\Q}{{\mathbb{Q}}}
\newcommand{\F}{{\mathbb{F}}}
\newcommand{\N}{{\mathbb{N}}}
\newcommand{\R}{{\mathbb{R}}}
\newcommand{\A}{{\mathbb{A}}}
\newcommand{\E}{{\mathbb{E}}}
\newcommand{\D}{{\mathbb{D}}}
\newcommand{\HH}{{\mathbb{H}}}
\newcommand{\C}{{\mathbb{C}}}
\newcommand{\trace}{\mbox{trace}}

\begin{document}

\title{
Boris Venkov's Theory of Lattices and Spherical Designs}

\author{Gabriele Nebe}
\address{
Lehrstuhl D f\"ur Mathematik, RWTH Aachen University,
52056 Aachen, Germany}
\email{nebe@math.rwth-aachen.de}


\maketitle

\section{Introduction.}

Boris Venkov passed away on November 10 2011 just 5 days before his 77th birthday. 
His death overshadowed the conference
``Diophantine methods, lattices, and arithmetic theory of quadratic forms''
November 13-18 2011 at the BIRS in Banff (Canada), where his  
important contributions to the theory of lattices, 
modular forms and spherical designs played a central role. 
This article gives a short survey of the mathematical work of Boris Venkov in this direction.

Boris Venkov's first work on lattices was a new proof \cite{V1} of the classification of even unimodular 
lattices in dimension 24 in  1978 which appeared as Chapter 18 of 
the book ``Sphere packings, lattices and groups'' \cite{SPLAG}.
This was the first application of the theory of spherical designs to lattices shortly after 
their definition in the fundamental work \cite{DGS} by Delsarte, Goethals, and Seidel. 
In the same spirit a combination of the theory of spherical designs with modular
forms allowed Venkov to prove that all layers of extremal
even unimodular lattices form good spherical designs. 
Since then, lattices became an important tool for the construction and investigation 
of spherical designs (see for instance \cite{BMV} or Section \ref{tight}).
Boris Venkov's work on the 
connection between lattices and spherical designs 
finally led
him to the definition of strongly perfect lattices. His
lecture series in Bordeaux and Aachen on this topic (see \cite{stperf}) initiated many 
fruitful applications of this theory, some of which are collected in \cite{Mart}. 
Strongly perfect lattices provide interesting examples of 
locally densest lattice, so called extreme lattices. 
The definition of strongly perfect lattices was very 
successful, for instance it allows to apply the theory of modular forms to show that 
all extremal even unimodular lattices of dimension 32 
are extreme lattices.
It also permits to apply representation theory of the automorphism group to show that a lattice is extreme. 
The notion of strong perfection has been generalized to other
metric spaces like Grassmanians or Hermitian spaces and also
to coding theory.

Boris Venkov spent a great part of his mathematical life 
visiting other universities. When I asked him whether he has a complete
list of his visits for proposing him for his Humboldt Research Award 
in 2007, he answered
``A complete list of my visits would be too long. It contains also exotic
visits like Tata Institut in Bombay, Universidad Autonoma in Mexico
or Universidad de Habana, Kuba.''
He had visiting professorships in 
Bonn (1989), Cambridge (1989), Geneva (1992, and 3 months every year since 1997), 
Lyon (1993), Paris (1994), Aachen (1994, 1996, ff.), Berlin (1995/96), 
Grenoble (1995, 1998), Bordeaux (1997), Fukuoka (2002), Kyoto (2006). 
All groups enjoyed interesting and fruitful discussions and productive 
collaborations with Boris Venkov.
Some of the resulting articles are given in the references;
for a complete list of Boris Venkov's papers I refer to 
the Zentralblatt or MathSciNet.

\section{Lattices, designs and modular forms.} 

\subsection{Extreme lattices.}

A classical problem asks for the densest packing of equal spheres in Euclidean space. 
Already in dimension 3 this turned out to be very hard; in this generality it was solved
by T. Hales 1998 with a computer based proof of the Kepler conjecture. 
The sphere packing problem becomes easier, if one restricts to lattice 
sphere packings, where the centers of the spheres form a group. 
The density function has only finitely many local maxima on the 
space of similarity classes of $n$-dimensional lattices, the so called 
extreme lattices. Korkine and Zolotareff and later Voronoi developed 
methods to compute all extreme lattices of a given dimension. The necessary
definitions are given in this section. 
Details and proofs may be found in the textbook \cite{Martinet}.

We always work in 
Euclidean $n$-space $(\R^n,(,))$ where
$(x,y) = \sum_{i=1}^n x_iy_i$ is the standard inner product with associated
quadratic form $$Q: \R^n \to \R, Q(x):= \frac{1}{2} (x,x) =\frac{1}{2} \sum _{i=1}^n x_i^2 .$$  
 
\begin{definition} 
\begin{itemize}
\item[(a)] A {\bf lattice} is the $\Z $-span  $L=\langle b_1 ,\ldots , b_n \rangle _{\Z} = $ \\ $  \{ \sum _{i=1}^n a_i b_i \mid a_i\in \Z \}$ 
of a basis $B=(b_1,\ldots , b_n)$ of $\R^n$. 
\item[(b)] The {\bf determinant} of $L$ is the square of the covolume of $L$ in $\R^n$ and 
can be computed as the determinant of a {\bf Gram matrix} 
$\det(L) = \det ((b_i,b_j))_{i,j=1}^n$. 
\item[(c)] $L$ is called {\bf integral} if $(\ell , m ) \in \Z $ for all $\ell,m\in L$. 
\item[(d)] $L$ is called {\bf even} if $(\ell ,\ell ) \in 2\Z $ 
and hence $Q(\ell ) \in \Z$ for all $\ell \in L$.
\item[(e)] The {\bf minimum} of $L$ is  $\min(L) = \min \{ (x,x) \mid 0\neq x \in L \} .$
We denote the set of 
minimal vectors by $\Min (L) := \{ x\in L \mid (x,x) = \min (L) \} .$
\item[(f)] The sphere packing density of $L$ is then proportional to the 
{\bf Hermite function} $\gamma (L) := \frac{\min(L)}{\det(L)^{1/n}} $. 
\item[(g)] A {\bf similarity} of norm
 $\alpha \in \R_{>0}$ is an element
$\sigma \in \GL_n(\R) $ with $(\sigma(x),\sigma(y)) = \alpha (x,y)$ for all 
$x,y\in \R^n$.  
 Two lattices $L$ and $M$ are called {\bf similar}, if there is a similarity 
$\sigma $ with $\sigma(L) = M$.
\item[(h)] The Hermite function $\gamma $ is well defined on similarity classes of lattices.
A lattice $L$ is called {\bf extreme}, if its similarity class realises a local
maximum of $\gamma $.
\end{itemize}
\end{definition}

The definition of extreme lattices goes back to Korkine and Zolotareff in the 1870s. 
They showed that extreme lattices are perfect, 
where a lattice $L$ is {\bf perfect}, if the projections onto the minimal vectors span the
space of all symmetric endomorphisms, i.e. 
$$\langle x x^{tr} \mid x\in \Min (L) \rangle = \R^{n\times n}_{sym} .$$
30 years later Voronoi gave an algorithm to enumerate all finitely many
similarity classes of perfect lattices  in a given dimension.
In the study of perfect lattices we may always restrict to integral lattices by
the following result:

\begin{theorem}
Any perfect lattice $L$ is similar to some integral lattice.
\end{theorem}

\begin{proof}
Let $L$ be some perfect lattice of minimum 1. Choose some basis $B$ of $L$ and 
let $F:=((b_i,b_j) )_{i,j=1}^n \in \R^{n\times n}$ be the Gram matrix. 
Then $\Min (L) = \{ \sum _{i=1}^nx_ib_i \mid x \in \Z^n, x^{tr} F x = 1 \} $. 
View this as a system of linear equations with rational coefficients
 on the entries of the Gram matrix $F$: 
\begin{equation}\label{Eqaaa} 
 x^{tr} F x = 1 \mbox{ for all } \sum _{i=1}^nx_ib_i\in \Min (L) \end{equation}
We show that $F$ is the unique solution of this system: 
Let $G$ be a second solution of \eqref{Eqaaa}. Then
$$ 0 = x^{tr} (F-G) x = \trace (x^{tr} (F-G) x ) = \trace (x x^{tr} (F-G) ) \mbox{ for all } x\in \Min (L) $$
so $F-G\in \R^{n\times n}_{sym}$ is perpendicular to $\langle x x^{tr} \mid x\in \Min (L) \rangle $
 with respect to the positive definite symmetric bilinear form $(A,B)\mapsto \trace(AB)$.
Since $L$ is perfect this yields $F-G = 0$, so $F$ is uniquely determined by 
\eqref{Eqaaa}. 
All coefficients of \eqref{Eqaaa} are integers, so the solution
is rational,
 $F\in \Q ^{n\times n} $. 
Multiplying by the common denominator yields an integral lattice that
is similar to $L$.
\end{proof}

All perfect lattices are known up to dimension 8.
Due to the existence of the famous {\bf Leech lattice} $\Lambda _{24}$ 
we also know the absolutely densest lattice of dimension 24 by work of 
Elkies, Cohn and Kumar.

\begin{center}{\bf The densest lattices.} \\ \vspace{0.2cm} 
\begin{tabular}{|c|c|c|c|c|c|c|c|c|c|}
\hline
dimension & 1 & 2 & 3 & 4 & 5 & 6 & 7 & 8 & 24 \\
\hline
$\# $ perfect & 1 & 1 & 1 & 2 & 3 & 7 & 33 & 10916 &  \\
\hline
$\# $ extreme & 1 & 1 & 1 & 2 & 3 & 6 & 30 & 2408 &  \\
\hline
densest & $\A _1$ & $\A _2$ & $\A_3 $ & $\D_4 $ & $\D_5 $ & $\E_6 $ & $\E_7 $ &
$\E_8 $ & $\Lambda _{24} $ \\
\hline
\end{tabular}
\end{center}

For a perfect lattice to be a local maximum of the Hermite function an additional
convexity condition is needed:

\begin{definition}
A lattice $L$ is called 
{\bf eutactic}, 
 if there are $\lambda _x >0$ such that $I_n = \sum _{x\in \Min (L)} \lambda _x (xx^{tr}) $.
It is called 
{\bf strongly eutactic}, if all $\lambda _x$ can be chosen to be equal. 
\end{definition}

\begin{theorem}(\cite[Theorem 3.4.6]{Martinet}) 
A lattice $L$ is extreme if and only if it is perfect and eutactic. 
\end{theorem}

Up to dimension 8, the densest lattices are similar to root lattices.

\begin{definition}
A lattice $L$ is called a {\bf root lattice}, if $L$ is even and 
$L = \langle \ell \in L \mid Q(\ell ) = 1 \rangle _{\Z }$.
\end{definition}

Any root lattice is a unique orthogonal sum of indecomposable root lattices. 
The orthogonally indecomposable root lattices are classified. They form 
two infinite series $\A_n$ ($n\geq 1$), $\D_n$ ($n\geq 4$) with three
exceptional lattices $\E_6$, $\E_7$, $\E_8$.
An important invariant attached to an indecomposable root lattice is
its {\bf Coxeter number} $h(L) := |\Min (L)| / \dim (L) $. 

$$
\begin{array}{|c|c|c|c|c|c|} 
\hline
 L & |\Min (L)| & h(L) & \det(L)  &  n \\
\hline
\A _n & n(n+1) & n+1 &  n+1  & \geq 1 \\
\D _n & 2 n (n-1) & 2(n-1) &  4 & \geq 4 \\
\E _6 & 72 & 12 &  3 &  6 \\
\E _7 & 126 & 18 &  2   & 7 \\
\E _8 & 240 & 30 &  1  &  8  \\
\hline
\end{array}
$$

Venkov's study of root lattices described in Section \ref{Nie} 
gave the first connection between Voronoi's characterisation of extreme lattices
and spherical designs. A guiding observation comes from the fact that 
indecomposable root lattices are strongly eutactic.

\subsection{Strongly eutactic lattices.} 

\begin{remark} \label{strongeut} 
A lattice $L$ is strongly eutactic 
if and only if there is some constant $c$ such that 
\begin{equation}\label{Eqstrongeut} 
\sum  _{x\in \Min (L)} (x,\alpha )^2 =  c (\alpha , \alpha ) 
 \mbox{ for all } \alpha \in \R^n
\end{equation}
Applying the Laplace operator
$\Delta _{\alpha }:= \sum_{i=1}^n\frac{\partial^2}{\partial _{\alpha _i}^2} $ to both sides of Equation \eqref{Eqstrongeut} one gets $c= \min(L) |\Min (L)| /n $.
\end{remark}

\begin{proof}
The equation \eqref{Eqstrongeut} reads as 
$$\sum _{x\in \Min (L)} \alpha ^{tr} x x^{tr} \alpha 
 = c \alpha ^{tr} \alpha $$ for all $\alpha \in \R^n$ 
and therefore is equivalent to $\sum _{x\in \Min (L)} x x^{tr}  = c I_n $.
\end{proof}

\begin{definition}\label{Harmt} 
The 
space of {\bf harmonic polynomials} of degree $t$ in $n$ variables is
$$\Harm _t := \{ p\in \R[x_1,\ldots x_n] \mid \deg (p ) = t  \mbox{ and } 
\sum _{i=1}^n \frac{\partial^2}{\partial x_i^2} p = 0 \} $$
So $\Harm _t$ is the kernel of the Laplace operator. 
\end{definition}

\begin{remark}\label{Harmsteut}
The harmonic polynomials of degree $2$ are linear combinations of 
\begin{equation} \label{EqDefpa} 
p_{\alpha } : x \mapsto (x,\alpha )^2 - \frac{1}{n} (x,x) (\alpha,\alpha ) \mbox{ for }
\alpha \in \R^n .
\end{equation}
So a lattice is strongly eutactic, if and only if 
$$\sum _{x\in \Min (L)} p(x) = 0 \mbox{ for all } p\in \Harm _2 .$$
\end{remark}

Root lattices are important in the classification of complex semisimple Lie algebras 
but also for the classification of finite reflection groups. 
Any root $\ell \in \Min (L)$ of a root lattice $L$ defines an automorphism, 
the {\bf reflection} along $\ell $
$$\sigma _{\ell }: L\to L : x \mapsto x - (x,\ell ) \ell ,$$
so $\sigma _{\ell } \in \Aut (L) = \{ \sigma \in O(\R^n) \mid \sigma (L ) = L \} $. 
The {\bf automorphism group} $\Aut (L)$ is a finite subgroup of $\GL_n(\R) $.

\begin{prop}\label{irredeut}
If $\Aut(L) $ is irreducible, then $L$ is strongly eutactic. 
\end{prop}

\begin{proof}
Let $G = \Aut (L)$. Since $G$ is real irreducible 
all 
$G$-invariant quadratic forms are scalar multiples of  
$Q = \frac{1}{2} \sum _{i=1}^n x_i^2$. 
Now $G$ permutes the vectors in $\Min (L) $ and so 
$$Q': \R^n \to \R , \alpha \mapsto \sum _{x\in \Min (L)} (x,\alpha )^2 $$ 
is a positive $G$-invariant quadratic form; so there is some $c \in \R _{>0}$
such that $Q' = c Q$. By Remark \ref{strongeut} this means that $L$ is strongly eutactic.
\end{proof}

\begin{cor}\label{rootsteut}
Let $L\leq \R^n$ be an indecomposable root lattice.
Then $\Aut (L)$ is irreducible and hence $L$ is strongly eutactic, so 
$\sum _{x\in \Min (L)} (x,\alpha )^2 = 2 h(L) (\alpha , \alpha ) $ 
for all $\alpha \in \R ^n$.
\end{cor}

\subsection{Extremal lattices.}

The notion of (analytic) extremality has first been defined for even unimodular lattices
and has then be generalized by Quebbemann \cite{Quebbemann} to modular lattices.
Roughly speaking, extremal lattices are lattices in some arithmetically defined family of
lattices for which the  density is as big as the  theory of modular forms allows it to be.
The idea is to translate arithmetic properties of the lattice $L$ into invariance 
conditions of its theta series $\theta _L(z)$ and to prove that $\theta _L (z)$ 
is some homogeneous element in a finitely generated graded ring (or module) of modular forms.
The knowledge of explicit generators then allows to derive \`a priori upper bounds on the
minimum of $L$. Details on this section can be found in the books  \cite{Ebeling}
and \cite{Serre}.

\begin{definition}
Let $L=\langle B \rangle _{\Z }\leq \R^n $ be a lattice.
\begin{itemize}
\item[(a)] The {\bf dual lattice} $L^{\#} := \{ x\in \R^n \mid (x,\ell ) \in \Z \mbox{ for all }
\ell \in L \} $ is the lattice spanned by the dual basis $B^*$. 
\item[(b)] $L$ is called {\bf unimodular} if $L=L^{\#}$.
\item[(c)] Let $L$ be an even lattice. 
Then the {\bf theta series} of $L$ is 
$$\theta _L := \sum _{\ell \in L} q^{Q(\ell )} = 1+\sum _{j=\min(L)}^{\infty }  a_j q^j $$
with $a_j = |A_j(L) |$ and $A_j(L) = \{ \ell \in L \mid Q(\ell ) = j \} $.
The substitution $q:=\exp(2\pi i z)$ then defines a holomorphic function
$\theta _L(z) = \sum _{j=0}^{\infty }  a_j \exp(2\pi i z)^j $ on the upper half plane 
$\HH := \{ z\in \C \mid \Im (z) > 0 \}$.
\end{itemize}
\end{definition}

In the following we will study {\bf even unimodular lattices}. They
correspond to positive definite regular integral quadratic forms $Q:L\to \Z $. 
The theory of quadratic forms shows that even unimodular lattices only 
exist, if the dimension $n$ is a multiple of $8$. 

By the periodicity of the exponential function,
the theta function of an even lattice is invariant under the substitution $z\mapsto z+1$.
The so called {\bf theta transformation formula} (\cite[Proposition 2.1]{Ebeling},
\cite[Proposition 16]{Serre})
 relates the theta series of the dual lattice 
$L^{\#} $ to $\theta _{L}$. 
In particular the theta function of an even unimodular lattice $L$ is a 
modular form of weight $n/2$ for the full modular group $\SL_2(\Z) $. 
For details on modular forms (including their definition) I refer to \cite{Ebeling}
or \cite{Serre}. 
The main result we need here is the following theorem describing the structure
of the graded ring of modular forms.

\begin{theorem}
Let $E_4$ and $E_6$ denote the normalized Eisenstein series of weight $4$ and $6$,
$$E_4 =  1 +240 \sum _{j=1}^{\infty } \sigma _3(j) q^j =\theta _{\E_8}
,\ E_6 = 1 - 504 \sum_{j=1}^{\infty} \sigma _5(j) q^j $$
where $\sigma _r(j)$ is the sum over the $r$-th powers of all divisors of $j$.
 Then the ring of modular forms for the full modular group is
$${\mathcal M}(\SL_2(\Z) )  = \C[E_4,E_6] $$ 
 the polynomial ring in $E_4$ and $E_6$.
\end{theorem}

So any modular form $f$ of weight $k$ has a unique expression as 
$$f=\sum _{4a+6b=k} c(a,b) E_4^aE_6^b \mbox{ with } c(a,b) \in \C  . $$
The vanishing order of $f$ at $z=i\infty $  (so $q=0$) defines a valuation on 
${\mathcal M}(\SL_2(\Z) )$ with associated maximal ideal ${\mathcal S}(\SL_2(\Z))$, the
space of {\bf cusp forms}. This is a principal ideal generated by 
$\Delta $ with 
\begin{equation} \label{EqDelta}
\Delta = \frac{1}{1728} (E_4^3-E_6^2) = 
q -24 q^2  + 252 q^3 - 1472 q^4 + \ldots  \end{equation}

\begin{theorem}
Let $L$ be an even unimodular lattice of dimension $n$.
Then $n$ is a multiple of $8$ and $\theta _L \in \C[E_4,\Delta ]_{n/2}$. 
\end{theorem}

For $4k=n/2$ the space 
 ${\mathcal M}_{4k}(\SL_2(\Z ))$ has a very nice basis.
$$\begin{array}{llllll} 
E_4^k  = & 1 + & 240 k q + & * q^2 +& \ldots \\
E_4^{k-3} \Delta  = & & \phantom{240} q + & * q^2 + & \ldots \\
E_4^{k-6} \Delta^2  =&  &  & \phantom{*} q^2 + & \ldots \\
\vdots  \\
E_4^{k-3m_k} \Delta ^{m_k} = & & \ldots & & q^{m_k} + & \ldots \\ 
\end{array}
$$
where $m_k=\lfloor \frac{n}{24} \rfloor = \lfloor \frac{k}{3} \rfloor $.

In particular ${\mathcal M}_{4k}(\SL_2(\Z))$  contains a unique form
$$f^{(k)}:= 1+ 0 q + 0 q^2 + \ldots + 0 q^{m_k} + a(f^{(k)}) q^{m_k+1} + b(f^{(k)}) q^{m_k+2} + \ldots $$
the {\bf extremal modular form} of weight $4k$.
If the minimum of an even unimodular lattice $L$ of dimension $n=8k$ is 
$\geq 2+2\lfloor \frac{n}{24} \rfloor$ then $\theta _L =f^{(k)}$ is equal to 
the extremal modular form.
Already Siegel has shown that the first nontrivial coefficient $a(f^{(k)})$ is always
a positive integer. In particular $\min (L)  =  2+2\lfloor \frac{n}{24} \rfloor$:

\begin{cor}\label{extremalbound}
Let $L$ be an even unimodular lattice of dimension $n$.
Then $$\min (L) \leq 2+ 2\lfloor \frac{n}{24} \rfloor .$$ The lattice $L$ 
is called {\bf extremal} if $\min (L) = 2+ 2\lfloor \frac{n}{24} \rfloor .$
\end{cor}

Since the second nontrivial coefficient $b(f^{(k)})$ of the extremal modular 
becomes negative for all $k\geq 20,408$
 there are no extremal even unimodular lattices
in dimension $n\geq 163,264$. 

\begin{center}
{\bf Extremal even unimodular lattices L$\leq \R ^n$}\footnote{The known
extremal lattices in the jump dimensions $24k$ are found in the
online database of lattices, http://www.math.rwth-aachen.de/$\sim$Gabriele.Nebe/LATTICES/}  \\ \vspace{0.2cm}
\begin{tabular}{|c|c|c|c|c|c|c|c|c|c|}
\hline
$n$ & 8 & 16 & 24 & 32 & 40 &  48 &  72 & 80 & $\geq 163,264$  \\
\hline
min(L) & 2 & 2 &  4 & 4 & 4   & 6 & 8  &  8 &  \\
\hline
number of & & & & & & & &  &  \\
extremal & 1 & 2 & 1 & $\geq 10^7$  & $ \geq 10^{51} $ & $\geq 3$  & $\geq 1$ & $\geq 4$ & 0 \\
lattices & & & & & & &  & & \\
\hline
\end{tabular}
\end{center}

The densest lattices in dimension 8 and 24 
and the densest known lattices in dimension 48 and 72 
are extremal even unimodular lattices.
As we will see in Section \ref{ModFormext} Venkov's theory of 
strongly perfect lattices allows to show that extremal even unimodular lattices
of dimension $n\equiv 0,8 \pmod{24} $ are extreme, i.e. realise a 
local maximum of the density function. 

\subsection{Venkov's classification of Niemeier lattices.}\label{Nie}

In 1968 Niemeier classified the even unimodular lattices of dimension 24.
Up to isometry there are 24 such lattices $L$ and they are distinguished by their
root sublattice
$$R(L) := \langle \ell \in L \mid Q(\ell ) = 1 \rangle _{\Z } .$$
In 1978 Boris Venkov \cite{V1} gave a more structural proof of Niemeier's list 
by showing the following Theorem 

\begin{theorem}\label{dim24} 
Let $L$ be an even unimodular lattice of dimension $24$. Then 
\begin{itemize}
\item[(a)] 
The root sublattice $R(L)$ is either 0 or has full rank. 
\item[(b)] 
The indecomposable components of $R(L)$ have the same Coxeter number.
\end{itemize}
\end{theorem}

The
possible root systems  are then found combinatorially from the classification
of indecomposable root systems and their Coxeter numbers:
$$
\begin{array}{l}
\emptyset,\ 24 \A_1,\  12 \A_2,\  8 \A_3,\  6 \A_4,\  4 \A_6,\  3 \A_8 ,\  2 \A_{12} ,\  \A_{24} , 
6 \D_4 ,\  4 \D_6,\\  3 \D _8,\  2 \D_{12},\  \D _{24} ,\
  4 \E_6,\  3 \E_8 , \
4 \A_5 \perp \D_4 ,\  2 \A_7 \perp 2 \D_5 ,\  2 \A_9 \perp \D_6,\\  \A_{15} \perp \D_9,   \ 
\E_8 \perp \D_{16} , 2 \E_7 \perp \D_{10} ,\  \E_7 \perp \A_{17} ,\  
\E_6\perp \D_7\perp \A_{11} \end{array} $$
If $R(L) \neq 0$, then 
$R(L) \leq L = L^{\#} \leq R(L)^{\#} $ so $L$ corresponds to
 a maximal isotropic subgroup of 
the discriminant group $R(L)^{\#}/R(L)$.
Going through all these 23 possibilities using algebraic coding theory one then
finds the following theorem.

\begin{theorem}
For each of the $23$ non zero root lattices listed above there is a unique
even unimodular lattice in dimension $24$ having this root sublattice.
\end{theorem}

The uniqueness of the Leech lattice, the unique even unimodular
lattice of dimension 24 with no roots is proved differently.
It follows for instance from the uniqueness of the Golay code,
but also by applying the mass formula.
$$\sum _{i=1}^{h} |\Aut(L_i)| ^{-1} = m_{2k} =
\frac{|B_{k}|}{2k} \prod _{j=1}^{k-1} \frac{B_{2j}}{4j} $$
where $L_1,\ldots , L_h $ represent the isometry classes
of even unimodular lattices in $\R ^{2k}$.

For the proof of Theorem \ref{dim24} we need the following result by
Hecke. 

\begin{theorem} \label{Hecke}
Let $L$ be an even unimodular lattice of dimension $n$ 
 and let 
$p\in \Harm _t$ be a harmonic polynomial of degree $t$
(see Definition \ref{Harmt}).
Then 
$$\theta_{L,p }:= \sum _{\ell \in L } p(\ell ) q^{Q(\ell )} \in {\mathcal M}_{n/2+t}(\SL_2(\Z) ) $$
is a modular form of weight $n/2+t$ for the full modular group.
\end{theorem} 

If $p=1$, then $\theta _{L,p} = \theta _L$.
For non constant homogeneous polynomials $p$ one has $p(0)=0$ and therefore 
$\theta _{L,p} \in {\mathcal S}_{n/2+t}(\SL_2(\Z))$ is a cusp form
and hence divisible by the form $\Delta $ from 
Equation \eqref{EqDelta}.

\begin{proof} (Theorem \ref{dim24})
Let $L$ be an even unimodular lattice of dimension 24 with $R(L) \neq 0$. For $\alpha \in \R^n$ let 
$p_{\alpha }\in \Harm_2$ be the harmonic polynomial 
defined in Equation \eqref{EqDefpa}. 
Then by Theorem \ref{Hecke} the theta series 
$$\theta _{L,p_{\alpha }} \in \Delta {\mathcal M}_{12+2-12} (\SL_2(\Z)) = \Delta {\mathcal M}_2(\SL_2(\Z)) = 0 $$
since there are no non zero modular forms of weight 2. 
But this implies that 
$$\sum_{x\in \Min(L)} p_{\alpha }(x) = 0 \mbox{ so } 
\sum _{x\in \Min (L)} (x,\alpha )^2 = \frac{2|\Min (L)|}{24} (\alpha , \alpha ). $$
In particular if $(x,\alpha ) = 0$ for all $x\in \Min (L)$ then $\alpha = 0$ and 
hence $R(L) ^{\perp } = 0$. 
\\
Now write
 $R(L) = R_1\perp \ldots \perp R_s $ with indecomposable root lattices $R_i$
of dimension  $n_i=\dim (R_i)$.
For $\alpha \in \langle R_i \rangle _{\R }$ we obtain
$$\sum _{x\in \Min (L)} (x,\alpha )^2  = 
\sum _{x\in \Min (R_i)} (x,\alpha )^2  = \frac{2|\Min (R_i)|}{n_i} (\alpha ,\alpha ) $$
by Corollary \ref{rootsteut} and Remark \ref{strongeut}.
Hence $h(R_i) = \frac{|\Min (R_i)|}{n_i} = \frac{|\Min (L)|}{24} $ for all $i$.
\end{proof}

\subsection{The Koch-Venkov invariant.}

Even unimodular lattices are fully classified up to dimension 24.
In dimension 32 the mass formula shows that there are more than 80 million
such lattices, 
 more than 10 million of which are extremal (\cite{King}). 
Nevertheless Koch (1988) and Venkov (\cite{V321}, \cite{V322}, \cite{V323})
 started to investigate 
32-dimensional even unimodular lattices. During these days it was not 
possible to algorithmically decide equivalence of 32-dimensional extremal lattices.

To understand the motivation of Koch and Venkov one should recall the
well known correspondence between framed unimodular lattices and  self-dual 
codes.

\begin{remark}\label{constructionA}
Let $L=L^{\#} \leq \R^n$ be a unimodular lattice and 
$F:=\{ v_1,\ldots , v_n \} \subset L $ be a {\bf p-frame}, 
i.e. a set of pairwise orthogonal vectors of norm $(v_i,v_i) = p$ for all $i$. 
Then any $\ell \in L$ is a unique sum 
$\ell = \sum_{i=1}^n a_i v_i $ with $a_i \in \frac{1}{p} \Z $ and 
$$C(L,F):= \{ (\overline{a}_1,\ldots , \overline{a}_n) \mid \sum a_iv_i \in L \} \leq \F_p^n $$ 
is a self-dual code.
Here $\overline{a} = a+\Z \in \frac{1}{p} \Z / \Z \cong \F_p$.
On the other hand, given some $C=C^{\perp}\leq \F_p^n$ and 
a $p$-frame $\{ v_1,\ldots , v_n \}$ the lattice 
$$L(C):= \{ \sum _{i=1}^n a_i v_i \mid 
 (\overline{a}_1,\ldots , \overline{a}_n ) \in  C , 
a_i\in \frac{1}{p}\Z  \} $$
is a unimodular lattice. 
$L(C)$ is even, if and only if $p=2$ and $C$ is doubly-even.
\end{remark}

Koch and Venkov define the {\bf defect} of an integral $n$-dimensional lattice
$L$ as $\delta (L):=n-s$,
 where $s$ is the maximal cardinality of a set of pairwise orthogonal
roots in $L$. 
So the lattices of defect 0 are exactly the lattices $L(C)$ for self-dual binary
codes $C$. 
Koch and Venkov show the following 

\begin{theorem}\label{deltane}
Let $L$ be an integral unimodular lattice of even dimension  $n$.
If $\delta (L) \leq 13$, then $\delta(L)$ is one of $0$, $8$, or $12$.
\end{theorem}

\begin{proof}
For the proof they use their notion of  perestroika of a lattice. 
Let $m=n-\delta(L)$ and $v_1,\ldots , v_m \in L$ be pairwise orthogonal 
roots. Then $\langle v_1,\ldots ,v_m ,2L \rangle /2L \leq L/2L$ is an isotropic
space and hence contained in some maximal isotropic space $M/2L$. 
The sublattice $M$ of $L$ is called a {\bf perestroika} of $L$. 
Since $n$ is even, the dimension of $M/2L$ is $\frac{n}{2}$ and 
hence $\sqrt{2}^{-1} M $ is a unimodular lattice containing the 
sublattice $\perp _{i=1}^m \frac{1}{\sqrt{2}} v_i \Z \cong \Z^m$.
So $M \cong \Z^m \perp N$ for some unimodular lattice $N$ of dimension
$\delta (L)$ of minimum $\geq 2$. By \cite{Kneser} there is no
such lattice $N$ of dimension $1,2,3,4,5,6,7,9,10,11,13$. 
If $\delta (L) =8$ then $N\cong \E _8$ and for $\delta (L) = 12$
the lattice $N$ is $\D_{12}^+$.
\end{proof}

Of course extremal lattices of dimension 32 have minimum 4, so they do not
contain any roots. Nevertheless the definition of defect is helpful here by
considering neighbors of the lattice. 

\begin{definition} (\cite{Kneser})
Two unimodular lattices $L$ and $M$ are called {\bf neighbors} if 
$L\cap M$ has index $2$ in $L$ (and hence also in $M$). 
\end{definition}

Kneser has shown that all neighbors $M$ of the unimodular lattice $L$
are of the form 
$$ M= L^{(v)} := L_v + \Z \frac{v}{2} \mbox{ with } 
L_v := \{ \ell \in L \mid (\ell,v) \in 2 \Z \} $$
for some $v\in L$ such that $(v,v) \in 4 \Z$. 
If $L$ is an even lattice then $M$ is even, if and only if $(v,v)\in 8\Z$.

\begin{remark}
Let $L$ be an even unimodular lattice with no roots and $v\in L$
such that $(v,v)\in 8\Z $.
If $x,y\in L^{(v)} \setminus L$ then $x\pm y\in L$. So any two roots
$x\neq \pm y \in L^{(v)}$ are orthogonal to each other and therefore 
 $L^{(v)}$ is an even unimodular lattice
with root system $(n-\delta (L^{(v)})) \A_1$. 
\end{remark}

\begin{definition} (Koch and Venkov)
Let $L$ be an extremal even unimodular lattice of dimension $32$. 
For $1\leq i\leq 32$ let $$f_L(i) := |\{ v\in L \mid (v,v) = 8 , \delta (L^{(v)}) =
32 -i \} | \mbox{ and } g_L(i):=\frac{f_L(i)}{2i}. $$
\end{definition}

Since $L^{(v)} = L^{(2w)}$ for any root $w\in L^{(v)}$ the function
$g_L$ takes only integer values. By Theorem \ref{deltane},  $g_L(i) = 0$ 
for $i=19,21,22,23,25,\ldots , 31$. 
Using modular forms Koch and Venkov \cite{KV2} prove the following equations for 
$g_L$:
$$\sum _{i=1}^{32} ig_L(i) = 2^53^55^217\cdot 733, \ 
\sum _{i=1}^{32} i^2g_L(i) = 2^{10}3^65^217^2, \ 
\sum _{i=1}^{32} i^3g_L(i) = 2^{14}3^75^217 . $$
They also compute the function $g_L$ for all lattices $L$ with $g_L(32)\neq 0$,
the neighbors $L=L(C)^{(w)}$ of the code-lattices $L(C)$ for one of the 5 
doubly-even self-dual extremal codes $C$.
In my diploma thesis I computed the function $g_L$ for those twelve lattices $L$
with $g_L(24)\neq 0$. 
The function $g_L$ seems to distinguish extremal 32-dimensional lattices.

\section{Lattices and spherical designs.} 
\subsection{Strongly perfect lattices.}\label{secstperf}
Most of the material in this section can be found in Boris Venkov's fundamental 
lecture notes \cite{stperf}. 
In 1977 Delsarte, Goethals, and Seidel \cite{DGS} define
the notion of spherical designs:

\begin{definition}
Let $X\subset S^{n-1}(m) := \{ x\in \R^n \mid (x,x) = m\} $ be some non-empty 
finite set. 
Then $X$ is called a {\bf spherical $t$-design}, if for all polynomials $p\in \R[x_1,\ldots , x_n]$
for degree $\leq t$ 
\begin{equation}\label{Eqspherdef} 
\frac{1}{|X|} \sum _{x\in X} p(x) = \int_{S^{n-1}(m)} p(x) dx .
\end{equation}
\end{definition}

Since the right hand side is the $O(\R^n)$-invariant inner product of $p$ with the constant function and
the homogeneous polynomials of degree $t$ are the orthogonal sum 
$$ \R[x_1,\ldots , x_n] _{t} = \Harm _t \perp Q \Harm _{t-2} \perp  Q^{2} \Harm _{t-4} \perp \ldots $$
the condition \eqref{Eqspherdef} is equivalent to 
\begin{equation}\label{Eqspherharm}
\sum_{x\in X} p(x) = 0 \mbox{ for all non constant harmonic polynomials } p \mbox{ of degree } \leq t .
\end{equation}

In particular Remark \ref{Harmsteut} says that a lattice $L$ is 
strongly eutactic, if and only if its minimal vectors form a spherical $2$-design.
Motivated by this observation Boris Venkov gave the following very fruitful definition.

\begin{definition} 
A lattice $L$ is {\bf strongly perfect} if its minimal vectors  form a 
spherical $4$-design. 
\end{definition}

Strongly perfect lattices provide interesting examples of locally densest lattices
as shown in the following theorem. In contrast to arbitrary extreme lattices, 
they can be classified in small dimensions using the combinatorics of their 
minimal vectors. 

\begin{theorem}
Strongly perfect lattices are strongly eutactic and perfect, so they are extreme.
\end{theorem} 

\begin{proof}
Let $L$ be a strongly perfect lattice.
Then the minimal vectors of $L$ form a spherical 2-design and hence $L$ is strongly 
eutactic by Remark \ref{Harmsteut}. 
 We need to show that $L$ is perfect, i.e. that 
$$\langle x x^{tr} \mid x\in \Min (L) \rangle = \R^{n\times n}_{sym } .$$
Note that any symmetric matrix $A\in \R^{n\times n}_{sym } $ defines a polynomial 
$p_A: \alpha \mapsto \alpha ^{tr} A \alpha $. 
Then $p_{x x^{tr}} (\alpha ) = (x,\alpha )^2$ and $$\trace(xx^{tr} A) = \trace(x^{tr} A x ) = p_A(x) .$$
Assume that 
$A \in \langle x x^{tr} \mid x\in \Min (L) \rangle^{\perp } $. Then  $p_A(x) = 0$ 
for all $x\in \Min (L)$. 
Since $\Min (L)$ is a spherical 4-design we obtain 
$$\int  _{S^{n-1}(\min(L))} p_A^2  = \frac{1}{|X|} \sum _{x\in X } p_A(x)^2 = 0 $$ 
which implies that $p_A=0$ and hence $A=0$.
\end{proof}

\begin{lemma} 
A lattice $L$ is strongly perfect, if and only if there is some constant $c$ 
such that 
$$  \sum _{x\in \Min (L)} (x,\alpha )^4 = c (\alpha , \alpha )^2 \mbox{ for all } \alpha \in \R^n .$$
\end{lemma}

As in Remark \ref{strongeut}, the constant $c$ is obtained by applying the
Laplace operator $\Delta _{\alpha } $ with respect to $\alpha  $ twice, 
$c= \frac{3\min(L)^2}{n(n+2)} |\Min(L)|$.
The  lemma only gives a polynomial condition of degree 4. 
Applying $\Delta _{\alpha }$ one obtains the condition of degree 2 from
Remark \ref{strongeut} that
characterises strongly eutactic lattices. Note that $\sum _{x\in \Min (L)} p(x) = 0$ 
for all homogeneous polynomials $p$ of odd degree since $\Min (L)$ is {\bf antipodal},
$\Min (L) = - \Min (L)$. 
Summarizing we obtain that 
 $L$ is strongly perfect if and only if for all $\alpha \in \R^n$ 
\begin{equation} \label{EqD4D2} 
\begin{array}{lll}
(D4) \ \  &
\sum _{x\in \Min (L)} (x,\alpha )^4  & = \frac{3|\Min (L)|m^2}{n(n+2)}  (\alpha ,\alpha ) ^2 \\
(D2) \ \  &
\sum _{x\in \Min (L)} (x,\alpha )^2  & = \frac{|\Min (L)|m}{n}  (\alpha ,\alpha )  
\end{array} 
\end{equation}

\begin{theorem}\label{minTyp}
Let $L$ be a strongly perfect lattice of dimension $n$. 
\\
Then $\min(L) \min(L^{\#} ) \geq (n+2)/3 .$
\end{theorem}

\begin{proof}
Let $\alpha \in \Min (L^{\#}) $ so $(\alpha,\alpha ) = \min(L^{\#})$.
 Then $(\alpha ,x) \in \Z $ for all $x\in \Min (L) $ and hence 
(D4)-(D2) =  $$
\sum _{x\in \Min(L)} \underbrace{(x,\alpha )^2 ((x,\alpha )^2 -1) }_{\geq 0} 
= \frac{|\Min(L)|\min(L)}{n} (\alpha , \alpha ) 
\underbrace{\big(\frac{3\min(L)(\alpha,\alpha )}{n+2} -1 \big)}_{\Rightarrow \geq 0} $$
Therefore $\frac{3\min(L)\min(L^{\#})}{n+2} \geq 1  $  and the theorem follows.
\end{proof}

\subsection{The classification of strongly perfect lattices} 

The formulas $(D4)$ and $(D2)$ from the last section allow to classify 
strongly perfect lattices of small dimension as well as strongly perfect integral 
lattices of small minimum.

\begin{theorem} \cite[Th\'eor\`eme 6.11]{stperf} 
The strongly perfect root lattices are 
$\A_1$, $\A_2 $, $\D_4$, $\E_6$, $\E_7$, and $\E_8$.
\end{theorem}

\begin{theorem} \cite[Th\'eor\`eme 7.4]{stperf} \label{min3}
The strongly perfect integral lattices of minimum $3$ are 
$O_1,O_7,O_{16},O_{22},O_{23} .$
\end{theorem}

The lattices $O_n$ of dimension $n$ are as follows
 $O_7 = \sqrt{2} \E_7^{\# }$,
$O_{16} = \langle \Lambda _{16} , x \rangle $, where
$\Lambda _{16}$ is  the Barnes-Wall lattice in dimension 16 and
$x\in \Lambda _{16} ^{\#} $ satisfies $(x,x)=3$.
$O_{23}$ is the unique unimodular lattice of minimum 3 and dimension 23,
$O_{22} = x^{\perp } $ for any minimal vector $x\in O_{23} $.

These classifications have been extended by J. Martinet  \cite[pp 135-146]{Mart} 
to integral lattices of higher minimum
by imposing stronger design conditions on the minimal vectors.

The strongly perfect lattices up to dimension 12 are all classified
(\cite{stperf}, \cite{dim10}, \cite{dim12}). 
It is believed that the lattices given in \cite[Tableau 19.1 and 19.2]{stperf} 
are the only strongly perfect lattices up to dimension 24.
In higher dimensions, the classifications get more and more involved.
To simplify them one might either impose stronger design conditions
(see for instance \cite{6des}) on the lattice or extra conditions on 
the dual lattice. Motivated by the fact that for most of the known 
strongly perfect lattices also the dual lattice is strongly perfect, 
we gave the following definition.

\begin{definition}
A lattice $L$ is called 
{\bf dual strongly perfect} if $L$ and $L^{\#}$ are both strongly perfect lattices.
\end{definition}

One method to show that a lattice $L$ is strongly perfect is to
use its automorphism group $G=\Aut(L)$.
If this group has no harmonic
invariant of degree $\leq 4$, then all  $G$-orbits are
spherical 4-designs (see Section \ref{rep}) and hence the lattice 
is strongly perfect. 
Since $\Aut(L) = \Aut(L^{\#}) $ such lattices are also
dual strongly perfect. A similar argument applies to lattices 
$L$ which are strongly perfect, because their harmonic theta series 
$\theta _{L,p} = 0$ for all harmonic $p$ of degree $2$ and $4$ (see Section \ref{ModForm}). 
The theta transformation formula then shows that also $\theta _{L^{\#},p} = 0$ 
and hence also the dual lattice is strongly perfect. 
In \cite{dim14} we showed that there is a unique dual strongly perfect lattice
 of dimension 14. 
The general method to classify all strongly perfect lattices in
a given dimension usually starts with a finite list of possible
pairs $(s,\gamma )$, where $s = s(L) = \frac{1}{2} |\Min (L) |$ is half
of the kissing number of $L$ and
$$\gamma  = \gamma '(L)^2 = \min (L) \min (L^{\#})$$
 the Berg\'e-Martinet invariant of $L$.
For both quantities there are good upper bounds known (\cite{Elkies}).
Note that $\gamma $ is just the product of the values of the
Hermite function on $L$ and $L^{\#}$.
Using the general equations \eqref{EqD4D2} of Section \ref{secstperf}
a case by case analysis allows either to exclude certain of the
possibilities $(s,\gamma )$ or to factor $\gamma = m \cdot r $ such that
rescaled to minimum $\min (L^{\#} ) = m$, the lattice $L^{\#} $ is integral
(or even) and in particular contained in its dual lattice $L$ (which is
then of minimum $r$).
For dual strongly perfect lattices we can use a similar argumentation
to obtain a finite list of possibilities $(s',\gamma )$ for
$s' = s(L^{\#} )$ and in each case a factorization $\gamma = m' \cdot r' $
such that $L$ is integral (or even) if rescaled to $\min (L) = m' $.
This gives the exponent (in the latter scaling)
$\exp (L^{\#} /L)  = \frac{m}{r'} $. We proceed 
either by a direct classification of all such lattices
$L$ or  use  modular forms to exclude
the existence of a modular form
$\theta _{L}$ of level $\frac{m}{r'} $ and weight $\frac{n}{2}$
starting with
$1+2sq^{m'/2} + \ldots $, such that its image under the Fricke involution starts
with $1+2s'q^{m/2} + \ldots $ and both $q$-expansions have non-negative
integral coefficients.
The classification of dual strongly perfect lattices up to dimension 17
is a PhD project of my student Elisabeth Nossek co-supervised by 
Boris Venkov.

\subsection{Application of group representations.} \label{rep}

Besides providing combinatorial tools for the
classification of certain locally densest lattices, the 
notion of  strongly perfect lattices opens to apply representation 
theory of finite groups but also the theory of modular forms (Section \ref{ModFormext})
to prove that certain lattices are extreme. 

Similar to Proposition \ref{irredeut} one shows the following Lemma.

\begin{lemma}
Let $G\leq \Aut (L)$ and assume that all homogeneous $G$-invariant polynomials
of degree $4$ are multiples of $Q^2$. 
Then $L$ is dual strongly perfect.
\end{lemma}

Together with Venkov we tried to apply this to obtain the minimum of 
the Thompson-Smith lattice of dimension 248:
Let $G = $Th be the sporadic simple Thompson group.
Then $G$ has a 248-dimensional rational representation $\rho : 
G \to O(\R^{248}) $.
Since $G$ is finite, $\rho (G)$ fixes a lattice
$L\leq \Q^{248}$.
 Modular representation theory tells us that for all
primes $p$ the $\F_pG$-module $L/pL$ is simple.
 Therefore $L=L^{\#} $ and $L$ is even
(otherwise the even sublattice $L_0$ of $L$ provides an $\F_2G$-submodule 
$L_0/2L < L/2L $). 
From the character table of $G$ one obtains that the space of $G$-invariant
homogeneous polynomials of
degree $2d$ is spanned by $Q^d$ for 
$d=1,2,3 $.
So all layers of $L$ form spherical 6-designs and in particular $L$
is strongly perfect.
Theorem  \ref{minTyp}  implies that 
$\min (L) \min(L^{\#} ) = \min(L) ^2 \geq \frac{248+2}{3}  > 83  $, so
$\min (L) \geq 10$.
Constructing the lattice $L$ one finds a vector 
 $v\in L$ with $Q(v) = 6$, so
$\min (L) \in \{ 10,12 \}$.

\begin{cor}
The minimum of the Thompson-Smith lattice is either $10$ or $12$.
\end{cor}

\section{Unimodular lattices} \label{ModForm} 

\subsection{Extremal even unimodular lattices are extreme.}\label{ModFormext}

Boris Venkov was the first who used 
the theory of modular forms to 
study designs supported by extremal even unimodular lattices
(see \cite{ExtDes}, \cite[Chapter 7, Theorem 23]{SPLAG}).
This was generalized to extremal modular lattices (in the sense of
Quebbemann \cite{Quebbemann}) by Bachoc and Venkov 
in \cite[pp 87-111]{Mart}.

\begin{theorem}
Let $L$ be an extremal even unimodular lattice of dimension $n=24a+8b$
with $b=0,1,2$. Then 
 all nonempty layers $A_j(L)$ are ($11-4b$)-designs.
\end{theorem} 

\begin{proof}
Since $L$ is extremal, its minimum is $2a+2$. Let $p\in \Harm _t$ be 
a harmonic polynomial of degree $t\geq 1$. Then 
$$\theta _{L,p} = \sum _{j=a+1}^{\infty } (\sum _{\ell \in A_j(L)} p(\ell ) ) q^j 
\in \Delta ^{a+1} {\mathcal M}_{4b-12+t} $$
Therefore $\theta _{L,p} = 0$ whenever $4b+t  < 12$, hence all layers 
$A_j(L) =\{ \ell \in L \mid Q(\ell ) = j \} $ form spherical ($11-4b$)-designs.
\end{proof}

\begin{cor}
 If $b=0$ or $b=1$ then $L$ is strongly perfect and hence extreme.
\end{cor}

 In particular all extremal even unimodular lattices of dimension $32$
are extreme.
O. King \cite{King} has shown that there are more than $10$ million  such lattices.
A complete classification is unknown and the theory of strongly perfect lattices
is the only known method to prove that all these lattices provide local 
maxima of the density function.

\subsection{Odd unimodular lattices and their shadow.} 

The theta series of an odd unimodular lattice is only a modular form for a 
subgroup of index 3 of the full modular group.
The upper bound on the minimum of an odd unimodular lattice $L\leq \R^n$ 
obtained by  the theory of modular forms in the same way as for even lattices
in Corollary \ref{extremalbound} is $$\min (L) \leq \lfloor \frac{n}{8} \rfloor + 1 .
$$
The only unimodular lattices where equality is achieved are 
${\bf Z}^n $ $(n=1,\ldots, 7$), $\E _8$, $\D_{12}^+$, $(\E_7\perp \E_7)^+$,
$\A_{15}^+$, $O_{23}$ and $\Lambda _{24}$ (see \cite[Chapter 19]{SPLAG}). 

Any odd unimodular lattice $L$ contains its {\bf even sublattice}
\begin{equation} \label{Eqevensub} 
L_0 := \{ \ell \in L \mid (\ell ,\ell ) \in 2\Z \} 
\end{equation} 
of index 2. The theta series of $L_0$ is 
$\theta _{L_0} = \frac{1}{2} (\theta _L(z) + \theta _L(z+1) ) $ and 
also $\theta _{L_0^{\#} } $ is obtained from $\theta _L$ using the 
theta transformation formula. 

\begin{definition}
Let $L$ be an odd unimodular lattice. Then
the {\bf shadow} of $L$ is
$S(L) := L_0^{\#} \setminus L $.
\end{definition}

Note that $S(L)$ is not a lattice but the union of the two cosets $\neq L$ of 
$L_0$ in $L_0^{\#}$. The theta series of $S(L)$ is obtained from the 
theta series of $L$ as 
$$\theta _{S(L)} (z) = S(\theta _L(z)) = (\frac{i}{z})^{n/2} 
\theta _{L}(-\frac{1}{z} +1 ) .$$
Using the fact that $S(\theta _L)$ also has non-negative integer coefficients 
Rains and Sloane \cite{RS} prove the following  theorem.

\begin{theorem} 
Let $L\leq \R^n$ be an odd unimodular lattice.
Then $\min (L) \leq  2+ 2\lfloor \frac{n}{24} \rfloor $ except for $n=23$, 
where this bound is $3$.
\end{theorem} 

A similar result holds for odd modular lattices. 

Any $v\in S(L)$ satisfies $(v,\ell ) \equiv Q(\ell ) \pmod{\Z } $, 
so $2 v$ is a {\bf characteristic vector} of $L$. 
By the theory of quadratic forms, the norm $(2v,2v)\equiv n  \pmod{8} $.
Define $\sigma (L) := 4 \min (S(L)) $ to be the minimal norm of a characteristic
vector in $L$.

Elkies proved that $\Z ^n$ is the only unimodular lattice
$L$  with $\sigma (L ) = n$.
Any unimodular lattice $L$ can be written uniquely as  $L=M\perp \Z^k$ 
with $M=M^{\#}$ of minimum $\geq 2$. Then $\sigma (L) = \sigma (M) + k$,
so one may assume that $\min (L) \geq 2$. 
Then Elkies
found the short list of lattices
$L$ of minimum $\geq 2$ with  $\sigma  (L ) = n-8$.
The largest possible dimension here is $n=23$ where the lattice
$L$ is the shorter Leech lattice $O_{23}$.
In \cite{shadow} 
we adapt the theory of
theta series with spherical coefficients to the shadow
theory of unimodular lattices
to study lattices $L$ with $\sigma (L) = n-16$.
If $\min (L) \geq 3$ then $n\leq 46$.
This bound is the best possible, because
$L = O_{23} \perp O_{23} $ satisfies $\dim (L) = 46$ and
$\sigma (L) = 46 -16 $ and this is the only such
lattice of dimension 46. In dimension 45 and 44 there are no such lattices of
minimum $\geq 3$. 

The combination of the minimum of the lattice and its shadow motivated Bachoc and
Gaborit to 
define  $s$-extremal lattices and codes.
Gaborit showed that a unimodular lattice $L\leq \R^n$ always satisfies 
$ 8 \min (L) +\sigma (L) \leq 8 + n $ unless $n=23 $ and $L=O_{23}$. 
Lattices achieving this bound are called {\bf s-extremal}. 
This notion has been generalized
 to modular lattices.

\subsection{Classification of odd unimodular lattices.} 

All unimodular lattices are classified up to dimension 25 \cite[Chapter 16,17]{SPLAG}.
Borcherds also 
showed that there is a unique unimodular lattice in dimension 26 
without roots. In higher dimensions the mass formula shows that 
there are too many unimodular lattices to classify them all. 
 Roland Bacher and Boris Venkov \cite[pp 212-267]{Mart}
 developed tools to classify only those odd unimodular lattices of dimension 
27 and 28 that have minimum $\geq 3$. They show that there are 3 such lattices
in dimension 27 and 38 such lattices in dimension 28. 
The correctness of their classification has later been also verified by 
the mass formulas in \cite{King}. 

Bacher's and Venkov's method of classification is as follows:
Let $L$ be a unimodular lattice of dimension 28 with minimum 3.
By the work by Elkies on lattices with long shadows, mentioned above,
$\sigma (L) $ is either 4 or 12.
This determines the two possibilities for the 
theta series of $L$. In particular $L$ contains vectors 
$v$ of norm $(v,v) = 4$. 
Each such vector $v\in L$ defines a neighbor $L^{(v)}  = \Z \perp M$ 
for some unimodular lattice $M$ of dimension 27 and of minimum $\geq 2$.
It can be shown that the root system of $M$ is $k\A _1$ and that 
there exists $v\in L$ such that $k\leq 4$. 
So it is enough to classify the 27-dimensional lattices $M $ with 
root system $k\A_1$ for $k=0,1,2,3,4$ and then construct $L$ as a 
neighbor of $M\perp \Z$. 
A clever counting argument using the symmetry of the neighboring graph 
makes the computation feasible.

\subsection{An application to coding theory.} 

The Bacher  Venkov classification of the unimodular lattices in dimension 28
without roots has a very nice application to the classification 
of extremal self-dual ternary codes of length 28 using Remark \ref{constructionA}. 
The paper \cite{HMV} shows that there are exactly 6931 such codes.
They correspond to pairs $(L,F)$ of 28-dimensional unimodular lattices
of minimum 3 and a 3-frame $F=\{v_1,\ldots ,v_{28} \} \subset L$. 
Since the minimum of the codes is $9$ one obtains a 
bijection between the set of equivalence classes of extremal self-dual ternary codes 
of length 28 and the set of pairs 
$ (L,F) $ of isometry classes of unimodular lattices $L$ of minimum 3 and 
representatives of the $\Aut(L)$-orbits of $3$-frames $F\subset L$.
Again theoretical arguments are needed to enable the enumeration of 
all frames with the computer.

\section{Tight spherical designs.} \label{tight}

The most interesting $t$-designs are those
of minimal cardinality. They have been studied by Bannai shortly after their definition 
in \cite{DGS}: 
If $t=2m$ is even, then any spherical $t$-design $X\subset S^{n-1}$  satisfies
$$|X| \geq {{n-1+m}\choose{m}} + 
{{n-2+m}\choose{m-1}} $$ and if 
$t=2m+1$ is odd then 
$$|X| \geq 2 {{n-1+m}\choose{m}} .$$ 
A $t$-design $X$ for which equality holds is called 
a {\bf tight} $t$-design.

Tight $t$-designs in $\R^n$ with $n\geq 3$ are very rare,
they only exist if $t\leq 5$ and for $t=7,11$. 
The unique tight 11-design is supported by the minimal vectors 
$\Min (\Lambda _{24})$ of the Leech lattice. 
The tight $t$-designs with $t=1,2,3$  are  also 
completely classified whereas  their classification for 
$t=4,5,7$ is still an open problem. 
It is known that the existence of a tight $4$-design in dimension 
$n-1$ is equivalent to the existence of a tight $5$-design in
dimension $n$, so the open cases are $t=5$ and $t=7$. 
It is also well known that tight spherical $t$-designs $X$ for odd
values of $t$ are antipodal, i.e. $X=-X$ (see \cite{DGS}).

There are certain numerical conditions on the dimension of such tight
designs. 
A tight $5$-design $X \subset S^{n-1}$ can only exist if 
either $n=3$ and $X$ is the set of 12 vertices of a regular 
icosahedron or
$n=(2m+1)^2-2$ for an integer $m$.
Existence is only known for $m=1,2$ and these designs are unique and 
given by the minimal vectors of $\E_7^{\#}$ resp. $M_{23}^{\#} [2] $
from \cite[Tableau 19.2]{stperf}.
Using lattices,  
Bannai, Munemasa and Venkov exclude the next two open cases $m=3,4$ as well as 
an infinity of other values of $m$ in the paper \cite{BMV}.

There are similar results for tight 7-designs. 
Such designs only exist if $n=3d^2-4$. The only known 
cases are $d=2,3$ and the corresponding designs are unique; 
they are given by the minimal vectors of the unimodular lattices $\E_8$ and 
$O_{23}$. 
The paper \cite{BMV} excludes the cases $d=4,5$ and also gives 
partial results on the interesting case $d=6$ which still
remains open.
The study of such designs in dimension $104$
 was part of our joint projects during Boris Venkov's last 
weeks in October and November 2011 in Aachen. 
Just to illustrate the connection with lattices a few arguments from \cite{BMV}
are recalled. 
So let $D = X \disj -X \subset S^{n-1}(d)$ be a tight spherical $7$-design
where $n=3d^2-4$, $d\in \N$. Then 
$$|X| = \frac{n(n+1)(n+2)}{6} \mbox{ and } (x,y) \in 0,\pm 1 \mbox{ for all } 
x\neq y \in X .$$
Let $L:=\langle X \rangle $. Then $L$ is an integral lattice of dimension $n$.
The design conditions (equation (D2) and (D4) from Section \ref{secstperf} and
the analogous equation (D6)) yield linear equations on the cardinalities
$$\begin{array}{ll} 
n_k(\alpha ) & := | \{ x\in X \mid (x,\alpha ) = \pm k \} | , (\alpha \in \R^n, k\geq 0) \\
\sum n_k(\alpha ) &  = |X| =  (1/2) (3d^2-4)(3d^2-2)(d^2-1)  \\
\sum k^2 n_k(\alpha ) &  = (1/2) (3d^2-2)(d^2-1)d (\alpha,\alpha ) \\
\sum k^4 n_k(\alpha ) &  = (3/2) (d^2-1)d^2 (\alpha,\alpha )^2 \\
\sum k^6 n_k(\alpha ) &  = (5/2) (d^2-1) d (\alpha,\alpha )^3. \end{array}$$
Assume that $\alpha \in \Min(L) \setminus D$.
Then $(\alpha ,x) \in \{ 0,\pm 1,\ldots , \pm \lfloor \frac{d}{2} \rfloor \} $ 
for all $x\in X$. 
In particular for $d\leq 7$ the $n_k(\alpha )$ are uniquely determined by 
the 4 equations. In all cases one obtains $n_2(\alpha ) < 0$ which is absurd. 
So $\min (L) = d$ and $D=\Min (L)$. 
Now let $\alpha \in L^{\#} $ be minimal in its class modulo $L$.
Again $(\alpha ,x) \in \{ 0,\pm 1,\ldots , \pm \lfloor \frac{d}{2} \rfloor \} $ 
for all $x\in X$. 
For $d =4$ and $ 5$ the system is overdetermined and one should find
$(\alpha ,\alpha )$ as rational root of the polynomial that determines 
$n_3(\alpha )$. But this polynomial has no non-zero rational roots.
Therefore $\alpha = 0$ and $L=L^{\#}$ is unimodular.
For $d=4$ this immediately yields a contradiction since then $L$ is 
even unimodular of dimension 44, which is not a multiple of 8. 
The case $d=5$ is more tricky. Here Venkov takes $\alpha \in L$ to be a 
characteristic vector of minimal norm. 
Then $(x,\alpha ) \in \{ \pm 1,\pm 3, \pm 5 \} $ again yields an 
overdetermined system on the $n_k(\alpha )$. 
One obtains a polynomial equation for $(\alpha ,\alpha )$ that has no 
rational solution. A contradiction. 

If $k$ is odd then $(k^2-1)(k^2-9)(k^2-25)$ is a multiple of $2^{10}3^25$.
This yields divisibility conditions on the norm of a characteristic vector.
In \cite{preprint} we show that 
 tight 7-designs with $d$ odd may only exist if $d\equiv \pm 1 \pmod{16}$ 
or $d\equiv \pm 3\pmod{32}$.

\section{Hecke operators.} 

In the previous section we have seen that one may apply 
modular forms, spherical designs and codes to construct and investigate
interesting lattices. 
We also saw application of lattices to the classification
of codes and tight designs. 
This final section reports on Venkov's ideas to apply the Kneser 2-neighbor 
graph of the Niemeier lattices to construct the action of 
certain Hecke operators on the space of Siegel modular forms spanned by
theta series \cite{Siemod12}. 
This has later been applied to other genera of modular lattices, 
including the genus of the Barnes-Wall lattice from \cite{Scharlau}, 
but also to genera of Hermitian lattices to construct Siegel cusp forms
as linear combinations of Siegel theta series. 
The transfer of this method to codes allowed me to define 
Hecke operators in coding theory which was an old 
question by Brou\'e. 

Let $L_1,\ldots , L_{24}$ represent the isometry classes of even
unimodular lattices in dimension 24. 
The Kneser 2-neighbor graph for these lattices has been computed by 
Borcherds for the purpose of classifying
odd unimodular lattices in dimension 24 (see \cite[Chapter 17]{SPLAG}).
The adjacency matrix 
$$K \in \Z ^{24\times 24} , K_{ij}:= |\{ M \mid M\cong L_j, [L_i : M\cap L_i] = 2 \}| $$
defines the action of a Hecke operator on the Siegel theta series (see
work of Yoshida and Walling).
The operator $K$ acts on the complex vector space 
$$V:=\{ \sum _{i=1}^{24} a_i [L_i] \mid a_i \in \C \} \cong \C^{24} $$
of formal linear combinations of the
Niemeier lattices. Taking the {\bf Siegel theta series} defines a linear mapping 
$$\Theta ^{(d)} : V \to {\mathcal M}_{12}(\Sp_{2d}(\Z)) , 
\sum _{i=1}^{24} a_i [L_i] \mapsto \sum _{i=1}^{24} a_i \Theta ^{(d)}(L_i) .$$
Let $V_d:=\ker(\Theta ^{(d)})$ be the kernel of $\Theta ^{(d)}$,
i.e. those linear combinations of lattices which have trivial
degree-$d$ Siegel theta series.
Then we get the filtration
\begin{equation} \label{Eqfilt}
V=:V_{-1} \supseteq V_0 \supseteq V_1 \supseteq \ldots \supseteq V_{m} = \{ 0 \}.
\end{equation}
The space $V$ has a natural positive definite Hermitian inner product 
defined by 
\begin{equation} \label{Eqscalp}
 \langle [ \Gamma ] , [ \Lambda ] \rangle :=
(\# \Aut(\Gamma ) ) \delta _{[ \Gamma ], [ \Lambda ] }.
\end{equation}
Let $Y_d:=V_{d-1} \cap V_d^{\perp}  $.
The space $S_d$ of degree-$d$ Siegel cusp forms that are linear combinations of 
Siegel theta series is then isomorphic to $Y_d\cong V_{d-1}/V_d $.
This yields the  orthogonal
decomposition 
\begin{equation} \label{Eqdec}
V = \bigoplus _{d =0}^m Y_d . \end{equation}
The purpose of \cite{Siemod12} is to compute this decomposition and 
therewith the spaces $S_d$.
The Kneser neighbor operator $K$ is self-adjoint with respect to the 
inner product \eqref{Eqscalp}. It respects the filtration \eqref{Eqfilt}
and hence also the decomposition 
\eqref{Eqdec} and therefore each space $Y_d$ has a basis consisting of 
eigenvectors of $K$. 
It turns out that $K$ has a simple spectrum, so it remains for each eigenvector
$e_1,\ldots ,e_{24}$ of $K$ to compute the number $d=w(i)$ such that $e_i\in Y_d$. 
This is a difficult problem which could not be solved completely. 
By computing some non-zero coefficient of $e_i$ one can always obtain 
upper bounds on $w(i)$. 
One important tool is the definition of an associative 
and commutative multiplication $\circ $ on $V$ for which the dual filtration of \eqref{Eqfilt}
behaves well, i.e. $V_n^{\perp}  \circ V_d^{\perp}  \subseteq V_{n+d}^{\perp}  $
for all $n,d$. 
The starting point was the cusp form $e_{24} \in Y_{12} $ constructed by 
Borcherds, Freitag, and Weissauer. 
We computed 
$e_i \circ e_j = A_{ij} e_{24} + \sum _{l=1}^{23} b_{ij}^l e_l $
with a non zero coefficient $A_{ij}$ for certain pairs $i,j$. 
This gave us the lower bound $w(i)+w(j) \geq 12$ which allowed us to 
obtain exact values for $w(i)$ and $w(j)$.
We could determine all $w(i)$ apart from one open conjecture which 
involves to prove that a certain linear combination of degree 9 Siegel 
theta series of weight 12 vanishes.

\end{document}